\documentclass[11pt,reno]{amsart} 
\usepackage{amssymb}
\usepackage{amsmath}
\usepackage{enumitem}
\usepackage{mathrsfs}
\usepackage[english]{babel}
\usepackage[all]{xy}
\usepackage{hyperref}
\usepackage{marginnote}
\usepackage{mathtools}

\usepackage{tikz-cd}
\usepackage{tikz}
\usepackage{fullpage}

\usepackage{stmaryrd}

\usepackage[margin=1.1in]{geometry}

\newtheorem{theorem}{Theorem}[section]
\newtheorem{lemma}[theorem]{Lemma}

\newtheorem{corollary}[theorem]{Corollary}
\newtheorem{conjecture}[theorem]{Conjecture}

\theoremstyle{definition}

\newtheorem{remark}[theorem]{Remark}

\newtheorem{definition}[theorem]{Definition}

\numberwithin{equation}{section} \numberwithin{figure}{section}

\DeclareMathOperator{\an}{an}

\DeclareMathOperator{\Kum}{Kum}
\DeclareMathOperator{\Sec}{Sec}
\DeclareMathOperator{\ConSec}{ConSec}

\newcommand*\ratmap{\mathbin{\tikz [baseline=-0.25ex,-latex, dashed, ->] \draw [densely dashed] (0pt,0.5ex) -- (1.3em,0.5ex);}}

\usepackage{color}

\definecolor{orange}{rgb}{1,0.5,0}

\title{Parshin's method and the geometric Bombieri-Lang conjecture}

\dedicatory{\hfill In memory of Jacob Murre }

\author{Finn Bartsch}
\address{Finn Bartsch \\
IMAPP Radboud University Nijmegen \\
PO Box 9010, 6500GL \\
Nijmegen, The Netherlands\\}
\email{f.bartsch@math.ru.nl }

\author{Ariyan Javanpeykar}
\address{Ariyan Javanpeykar \\ 
IMAPP Radboud University Nijmegen \\
PO Box 9010, 6500GL\\
 Nijmegen, The Netherlands}
\email{ariyan.javanpeykar@ru.nl }



\begin{document}

\begin{abstract}
In this short survey, we explain Parshin's proof of the geometric Bombieri--Lang conjecture, and show that it can be used to give an alternative proof of Xie--Yuan's recent resolution of the geometric Bombieri--Lang conjecture for projective varieties with empty special locus and admitting a finite morphism to a traceless abelian variety.
\end{abstract}
\maketitle

\thispagestyle{empty}

\section{Introduction}
In \cite{XieYuan1,XieYuan2}, Junyi Xie and Xinyi Yuan proved the geometric Bombieri--Lang conjecture for projective varieties with a finite morphism to an abelian variety.  In this short note we reprove part of their main result using an old ``anabelian'' method of Parshin \cite{ParshinLang, SzamuelyParshin}.
Our motivation for writing this note is to draw attention to Parshin's anabelian method and to the recent seminal work of Xie--Yuan.
 
Let $K /\mathbb{C}$ be a finitely generated field extension of transcendence degree one. Let $B$ be a smooth projective curve over $\mathbb{C}$ with function field $K(B)=K$. Let $X$ be a projective geometrically integral variety over $K$, and let $\mathcal{X} \to B$ be a projective model for $X$ over $B$.

We are interested in the following three notions. The first one is the ``Mordell property'' over function fields.  (In this paper we do not discuss the original version of the ``Mordell property'' over number fields considered  in  \cite{Lang2}.)

\begin{definition}
We say that $X$ is \emph{Mordellic (over $K$)} if $X(L)$ is finite for every finite field extension $L/K$. 
\end{definition}

Note that being Mordellic formalizes the property that $X$ ``has only finitely-many rational points'', and that it is equivalent to the finiteness of $T$-points of $\mathcal{X}$, where $T$ runs over every smooth projective curve over $\mathbb{C}$ with a finite morphism $T \to B$.

The property of $X$ having only finitely many rational points excludes the presence of abelian varieties (and more generally algebraic groups) inside $X$.

\begin{definition}
A proper geometrically integral variety $V$ over a field $k$ is \emph{groupless} if, for every abelian variety $A$ over $\overline{k}$, every morphism $A \to V_{\overline{k}}$ is constant.
\end{definition}

If $V$ is a proper geometrically integral variety over a field $k$ of characteristic zero, then the following are equivalent \cite[\S 3]{JXie}.
\begin{enumerate}
\item $V$ is groupless.
\item For every field extension $L/k$, the variety $V_L$ is groupless over $L$.
\item For every connected algebraic group $G$ over $\overline{k}$, every rational map $G\ratmap V_{\overline{k}}$ is constant.
\end{enumerate}

Note that the term ``groupless'' is motivated by the third characterization above. 
In more classical terms, a variety $V$ over a field $k$ is groupless if its ``special locus'' is empty. More precisely, define $\mathrm{Sp}(V)$ to be the closure of 
\[
\bigcup_{f} \mathrm{Im}(f)
\] in $V$,
where the union runs over all abelian varieties $A$ over $k$ and all non-constant rational maps $f \colon A \ratmap V$. With this definition, we see that $V$ is groupless if and only if $\mathrm{Sp}(V_{\overline{k}})$ is empty.

If $V$ is a smooth projective geometrically connected curve over $k$, then $V$ is groupless if and only if it has genus at least two. More generally, if $V$ is a closed subvariety of an abelian variety $A$ over a field $k$, then $\mathrm{Sp}(V)$ equals the union of all translates of positive-dimensional abelian subvarieties of $A$ contained in $V$. In particular, for a closed subvariety $V$ of an abelian variety $A$ over an algebraically closed field $k$, we see that $V$ is groupless if and only if $V$ does not contain the translate of any positive-dimensional abelian subvariety of $A$.

Recall that $K/\mathbb{C}$ is a finitely generated field extension of transcendence degree one. For every abelian variety $A$ over $K$, there is a finite field extension $L/K$ such that $A(L)$ is dense. Thus, if $X$ is Mordellic over $K$, then it is groupless.

Clearly, if $X$ is Mordellic and $V \to X$ is a non-constant morphism, then $V(K)$ is not dense. In other words, $X$ does not admit maps from varieties with many rational points. Besides abelian varieties, other examples of varieties with many $K$-points are provided by ``constant'' (positive-dimensional) varieties. Taking into account that $X(L)$ should be finite for \emph{every} finite extension $L/K$, we are led to the following notion of isotriviality.
 
\begin{definition}
We say that $X$ is \emph{highly non-isotrivial} if, for every variety $V$ over $\mathbb{C}$, every morphism $V\otimes_{\mathbb{C} } \overline{K} \to X\otimes_{K} \overline{K}$ is constant. (Note that $X$ is highly non-isotrivial if and only if, for every smooth projective curve $C$ over $\mathbb{C}$, every morphism $C\otimes_{\mathbb{C}} \overline{K} \to X\otimes_K {\overline{K}}$ is constant.)  
\end{definition}

It is not hard to show that if $X$ is Mordellic, then $X$ is highly non-isotrivial. In summary, we have the following easy implication:
\[
X \textrm{ Mordellic} \implies X \textrm{ groupless and highly non-isotrivial}
\]
A special case of the geometric Bombieri--Lang--Vojta conjecture says that the converse of this implication holds, i.e., the two obvious obstructions to being Mordellic are actually the only obstructions:

\begin{conjecture}[Geometric Bombieri--Lang]
If $X$ is groupless and highly non-isotrivial, then $X$ is Mordellic.
\end{conjecture}

This was previously known for curves \cite{ManinMordell, Grauert}, closed subvarieties of abelian varieties \cite{BuiumMordell, RaynaudMordell}, smooth projective varieties with ample cotangent bundle \cite{NoguchiMordell, MartinDeschamps}, and hyperbolic normal complex spaces with a hyperbolic embedding condition on the compactification \cite{Nogg85, Nogg92} (we stress that the hyperbolic embedding condition in \emph{loc. cit.} is automatically satisfied in the case of curves, so that this gives an alternative proof of the geometric Bombieri--Lang conjecture in this case). It was resolved by Xie--Yuan for $X$ assuming $X$ maps finitely to an abelian variety \cite{XieYuan1, XieYuan2}.

\subsection{The work of Xie--Yuan}
Suppose that $X$ is groupless and highly non-isotrivial, and suppose that there is a finite morphism from $X$ to an abelian variety $A$ over $K$. Let us explain the key ideas of Xie--Yuan for proving that the set $X(K)$ is finite (hence that $X$ is Mordellic).

A natural strategy for proving the finiteness of $K$-points on a groupless highly non-isotrivial projective variety is by establishing a height bound. Let $\mathcal{L}$ be an ample line bundle on $\mathcal{X}$. Assume that there is a constant $c$ depending only on $\mathcal{X} \to B$ and $\mathcal{L}$ such that, for every $\sigma\in \mathcal{X}(B) = X(K)$, the height $\deg \sigma^\ast \mathcal{L}$ of $\sigma$ (with respect to $\mathcal{L}$) is at most $c$. To now deduce the finiteness of $X(K)$, Xie--Yuan invoke the work of Ueno (for closed subvarieties of abelian varieties) and Kawamata (for varieties with a finite morphism to an abelian variety) in \cite{KawamataChar,Ueno} to deduce from the grouplessness of $X$ that $X$ is of general type. Then, following arguments of Noguchi \cite{Nogg85}, the finiteness of $X(K)$ is deduced from the existence of a height bound. (The gist of the argument is contained in \cite[Theorem~2.8]{XieYuan1}.) Thus, this (essentially) reduces the geometric Bombieri--Lang conjecture to proving a height bound (in this setting).

One of Xie and Yuan's novel insights is that Brody's reparametrization lemma implies a bound on the ``partial height'' of a $K$-point (defined by taking the ``degree'' of $\sigma^\ast \mathcal{L}$ on some open disk contained in $B^{\an}$), assuming $\mathcal{X} \to B $ has hyperbolic fibres over a closed disk inside $B^{\an}$; see \cite[Theorem~2.7]{XieYuan1}. The latter result would suffice to prove the above geometric Bombieri--Lang conjecture if we knew how to relate the partial height and (standard) height. This is however not known in general, and is formulated as the ``non-degeneracy conjecture'' \cite[Conjecture~2.3]{XieYuan1}. In the case that $X$ admits a finite morphism to a ``traceless'' abelian variety, this conjecture can be proven using properties of the canonical N\'eron-Tate height on an abelian variety; see \cite[Theorem~3.6]{XieYuan1}. This then leads to one of the main results in Xie--Yuan's work:

\begin{theorem}[Xie--Yuan, special case] \label{thm:xy}
Let $B$ be a smooth curve over $\mathbb{C}$ with function field $K=K(B)$. Let $X$ be a geometrically integral groupless projective variety over $K$ such that there is an abelian variety $A$ over $\overline{K}$ of $\overline{K}/\mathbb{C}$-trace zero and a finite morphism $X_{\overline{K}}\to A$. Then, for every finite field extension $L/K$, the set $X(L)$ is finite.
\end{theorem}

The assumption that $A$ has $\overline{K}/\mathbb{C}$-trace zero means that for every abelian variety $A_0$ over $\mathbb{C}$, every morphism $A_{0,\overline{K}}\to A$ of abelian varieties over $\overline{K}$ is trivial (see \cite{Conrad} for a precise definition). In particular, by standard properties of the Albanese variety, the second assumption in Theorem \ref{thm:xy} implies that $X$ is highly non-isotrivial. Thus, Theorem \ref{thm:xy} is indeed a special case of the geometric Bombieri--Lang conjecture.

We note that the following improvements can be made in Theorem \ref{thm:xy} (see, for example, \cite[\S 4.2]{XieYuan1}):
\begin{enumerate}
\item One may replace $\mathbb{C}$ by an algebraically closed field $k$ of characteristic zero.
\item One may replace $B$ by any integral variety over $k$.
\item One may replace ``finite field extension $L/K$'' by ``finitely generated field extension $L/K$''.
\end{enumerate}
 
In this note, we only prove the geometric Bombieri--Lang conjecture using Parshin's method for certain projective varieties with an \emph{empty} special locus (i.e., Theorem \ref{thm:xy}). It is most noteworthy that Xie--Yuan resolve the geometric Bombieri--Lang conjecture in \emph{its entirety} for projective varieties whose Albanese map is generically finite onto its image. For example, their work leads to the following result \cite[Theorem~1.1]{XieYuan2}.

\begin{theorem}[Xie--Yuan]
Let $B$ be a smooth curve over $\mathbb{C}$ with function field $K=K(B)$. Let $X$ be a geometrically integral projective variety over $K$. Suppose that there is an abelian variety $A$ over $\overline{K}$ of $\overline{K}/\mathbb{C}$-trace zero and a finite morphism $X_{\overline{K}}\to A$. 	
Then, for every finite field extension $L/K$, the set $X(L)\setminus \mathrm{Sp}(X_{\overline{K}})$ is finite.
\end{theorem}

Applying Parshin's method to reprove this result would require some further developments in our understanding of pseudo-hyperbolic varieties (i.e., varieties which are hyperbolic modulo a proper closed subset).

\section{Groupless and hyperbolic}
If $k$ is a field, then a variety over $k$ is a finite type separated geometrically integral scheme over $k$. If $X$ is a locally finite type scheme over $\mathbb{C}$, then we let $X^{\an}$ denote its associated complex-analytic space \cite[Expos\'e~XII]{SGA1}.

Let us start with the following result on the behaviour of grouplessness in fibres.

\begin{theorem} \label{thm:groupless_families}
Let $k$ be an algebraically closed field of characteristic zero. Let $\mathcal{X} \to B$ be a proper surjective morphism of varieties over $k$ where $B$ is a smooth curve. Then, the set of $b$ in $B$ such that $\mathcal{X}_b$ is not groupless is a countable union of closed subschemes of $B$.
\end{theorem}

This result is well-known and can be proved using a standard Hom-scheme argument already alluded to by Parshin in his proof of \cite[Proposition 2]{ParshinLang}; see also the proof of \cite[Theorem~1.6]{JVez}. For the sake of completeness, we note that the assumption that $B$ is a curve can be removed; this requires the additional \cite[Theorem~1.5]{JVez}, as explained in \cite[\S1.2]{vBJK}.

\begin{definition}
Let $X$ be a complex-analytic space. We follow \cite{KobayashiBook} and say that $X$ is \emph{hyperbolic} (resp. \emph{complete hyperbolic}) if the Kobayashi pseudometric $d_X$ is a metric (resp. complete metric) on $X$. A variety $X$ over $\mathbb{C}$ is \emph{hyperbolic} (resp. \emph{complete hyperbolic}) if $X^{\an}$ is hyperbolic (resp. complete hyperbolic).  
\end{definition}

\begin{remark}\label{remark:RS_is_complete} 
Let $X$ be a connected hyperbolic Riemann surface. Let us explain why $X$ is complete hyperbolic. First, by \cite[Theorem~3.2.1]{KobayashiBook}, the Kobayashi metric induces the topology of $X$. Thus, if $X$ is a compact Riemann surface, the complete hyperbolicity of $X$ follows from the fact that a compact metric space is complete. To deal with the case that $X$ is not compact, we note that by \cite[Theorem~3.2.8]{KobayashiBook}, a complex-analytic space is complete hyperbolic if and only if its universal covering is so. As all hyperbolic Riemann surfaces have the open unit disk $\mathbb{D}$ as their universal covering, the claim follows.
\end{remark}

The fact that hyperbolicity is Euclidean-open in proper families of complex spaces is a well-known result due to Brody; see \cite[Theorem~3.11.1]{KobayashiBook} (or the original \cite{Brody}).

\begin{lemma}[Brody] \label{lemma:open}
Let $\mathcal{X} \to B$ be a proper surjective morphism of connected complex-analytic spaces. If $b \in B$ is a point such that the fibre $\mathcal{X}_b$ is hyperbolic, then there is a Euclidean-open neighbourhood $U \subseteq B$ of $b$ such that, for every $u$ in $U$, the fibre $\mathcal{X}_{u}$ is hyperbolic. 
\end{lemma}

The following result was proven by Parshin \cite[Proposition~2]{ParshinLang} for subvarieties of abelian varieties. Parshin's line of reasoning can be adapted to prove the following more general result.

\begin{corollary}[Parshin + $\varepsilon$] \label{cor:parshin_plus_epsilon}
Let $\mathcal{X}\to B$ be a proper surjective morphism of varieties over $\mathbb{C}$ with $B$ a smooth projective curve. Suppose that there is a point $b_0$ in $B(\mathbb{C})$ such that $\mathcal{X}_{b_0}$ is groupless, and a dense open subset $U \subseteq B$ such that the set of $b$ in $U(\mathbb{C})$ with $\mathcal{X}_b$ groupless equals the set of $b$ in $U(\mathbb{C})$ with $\mathcal{X}_b$ hyperbolic. Then, the set of $b$ in $B(\mathbb{C})$ such that $\mathcal{X}_b$ is not hyperbolic is a Euclidean-closed countable subset of $B^{\an}$. 
\end{corollary}
\begin{proof}
By Theorem \ref{thm:groupless_families}, the set of $b$ in $B(\mathbb{C})$ with $\mathcal{X}_b$ not groupless is a countable union of Zariski-closed subsets of $B(\mathbb{C})$. As it does not contain $b_0$, it is thus countable. Consequently, by the defining properties of $U$, the set of points $b$ in $B(\mathbb{C})$ with $\mathcal{X}_b$ not hyperbolic is countable as well. We conclude by Lemma \ref{lemma:open}.
\end{proof}

\begin{definition}
Let $Z\subset X$ be a subset of a complex-analytic space $X$. We say that $X$ is \emph{Brody hyperbolic modulo $Z$} if every nonconstant holomorphic map $\mathbb{C}\to X$ lands in $Z$. We say that $X$ is \emph{pseudo-Brody hyperbolic} if there is a proper closed subset $Z \subsetneq X$ such that $X$ is Brody hyperbolic modulo $Z$. We say that $X$ is \emph{Brody hyperbolic} if $X$ is Brody hyperbolic modulo the empty subset.
\end{definition}

We will need the following consequence of Yamanoi's work on varieties of maximal Albanese dimension \cite{YamAb}. 

\begin{theorem}  \label{thm:Yamanoi}
Let $X$ be a projective variety over $\mathbb{C}$. If $X$ is groupless and admits a finite morphism to an abelian variety, then $X$ is hyperbolic.
\end{theorem}
\begin{proof}
Recall that, for $V$ a projective variety over $\mathbb{C}$, the special locus $\mathrm{Sp}(V)$ is defined to be $\cup_{A,f} \mathrm{Im}(f)$, where the union runs over all abelian varieties $A$ over $\mathbb{C}$ and all nonconstant rational maps $f \colon A \ratmap V$. By \cite[Corollary~1]{YamAb} and \cite[Theorem~13]{KawamataChar}, for $V$ a smooth projective variety of maximal Albanese dimension over $\mathbb{C}$ (i.e., the Albanese map of $V$ is generically finite onto its image), the following statements hold.
\begin{enumerate}
\item $V$ is of general type if and only if $\mathrm{Sp}(V) \neq V$.
\item If $V$ is of general type, then $V^{\an}$ is Brody hyperbolic modulo (the proper closed subset) $\mathrm{Sp}(V)^{\an}$. 
\end{enumerate}
Let $Y \subseteq X$ be a closed subvariety. Let $\widetilde{Y}\to Y$ be a resolution of singularities. Note that $\widetilde{Y}$ is a smooth projective variety of maximal Albanese dimension. Since $X$ is groupless, we have that $\mathrm{Sp}(\widetilde{Y}) \neq \widetilde{Y}$. Thus, by the aforementioned results of Kawamata and Yamanoi, we have that $\widetilde{Y}^{\an}$ is Brody hyperbolic modulo the proper closed subset $\mathrm{Sp}(\widetilde{Y})^{\an}$. It follows that $\widetilde{Y}^{\an}$ is pseudo-Brody hyperbolic. In particular, since pseudo-Brody hyperbolicity is a birational invariant amongst projective varieties, we have that $Y^{\an}$ is pseudo-Brody hyperbolic. 
Since $Y$ was chosen arbitrarily, it follows that $X^{\an}$ is Brody hyperbolic. Since $X^{\an}$ is a Brody hyperbolic compact complex-analytic space, it follows from Brody's theorem that $X^{\an}$ is hyperbolic \cite{Brody}.
\end{proof}

\begin{corollary}\label{cor:Yamanoi_groupless_family}
Let $\mathcal{X} \to B$ be a proper surjective morphism of varieties over $\mathbb{C}$ with $B$ a smooth projective curve. Assume that $\mathcal{X}_{K(B)}$ admits a finite morphism to an abelian variety $A$ over $K(B)$, where $K(B)$ denotes the function field of $B$. If $\mathcal{X}_{K(B)}$ is groupless, then the set of $b$ in $B(\mathbb{C})$ for which $\mathcal{X}_b$ is not hyperbolic is a Euclidean-closed countable subset of $B^{\an}$.
\end{corollary}
\begin{proof} 
Let $U \subseteq B$ be a dense open subset such that there is an abelian scheme $\mathcal{A} \to U$ with $\mathcal{A}_{K(B)} $ isomorphic to $A$ over $K(B)$. Shrinking $U$ if necessary, we may spread out the finite morphism $\mathcal{X}_{K(B)} \to A$ to obtain a finite morphism $\mathcal{X}_U \to \mathcal{A}_U$ of $U$-schemes. Thus, by Theorem \ref{thm:Yamanoi}, we see that for every $b$ in $U(\mathbb{C})$, the fibre $\mathcal{X}_b$ is groupless if and only if it is hyperbolic. By Theorem \ref{thm:groupless_families}, the set of $b$ in $B$ for which $\mathcal{X}_b$ is not groupless is a countable union of Zariski-closed subsets of $B$. Since $X_{K(B)}$ is groupless, the set of $b$ in $B$ for which $\mathcal{X}_b$ is not groupless does not contain the generic point. Thus, we see that there is a point $b_0$ in $B(\mathbb{C})$ for which $\mathcal{X}_{b_0}$ is groupless. Consequently, Corollary \ref{cor:parshin_plus_epsilon} applies and we conclude. 
\end{proof}

\section{Parshin's method}
Any finiteness statement boils down to some other finiteness statement. The crux of Parshin's method for proving finiteness of rational points is the following finiteness property of fundamental groups of hyperbolic varieties:

\begin{lemma}[Fundamental groups]\label{lemma:parshin_topological} 
Let $\mathcal{X}$ be a complete hyperbolic complex-analytic space, let $x \in \mathcal{X}$ and let $\varepsilon$ be a real number. Then, the set of elements of $\pi_1(\mathcal{X},x)$ which can be represented by a loop of Kobayashi-length at most $\varepsilon$ is finite.
\end{lemma} 
\begin{proof} 
(This is \cite[Lemma~2]{ParshinLang}; see also \cite[Lemma~4.5]{SzamuelyParshin}.)
Let $p \colon \widetilde{\mathcal{X}}\to \mathcal{X}$ be the universal cover of $\mathcal{X}$. Let $\tilde{x}$ be a point of $\widetilde{\mathcal{X}}$ lying over $x$. Then, a loop in $\mathcal{X}$ at $x$ lifts to a path in $\widetilde{\mathcal{X}}$ from $\tilde{x}$ to a point $y$ in $p^{-1}(x)$, and the class of the loop in $\pi_1(\mathcal{X},x)$ is determined by the point $y$. As the map $p$ is a local isometry for the Kobayashi metrics on $\widetilde{\mathcal{X}}$ and $\mathcal{X}$ (see \cite[Theorem~3.2.8]{KobayashiBook}), we obtain a bijection from the set of elements of $\pi_1(\mathcal{X},x)$ representable by loops of length at most $\varepsilon$ to the set 
\[ \{ y \in p^{-1}(x)~|~d_{\widetilde{\mathcal{X}}} (\tilde{x},y) \leq \varepsilon \}. \]

By \cite[Theorem 3.2.8]{KobayashiBook}, we see that $\widetilde{\mathcal{X}}$ is complete hyperbolic. Thus, combining \cite[Proposition~1.1.9]{KobayashiBook} with \cite[Theorem~3.1.15]{KobayashiBook} and \cite[Theorem~3.2.1]{KobayashiBook}, we see that the closed ball
\[ \{ y \in \widetilde{\mathcal{X}}~|~d_{\widetilde{\mathcal{X}}}(\tilde{x},y) \leq \varepsilon \} \]
of radius $\varepsilon$ is a compact subset of $\widetilde{\mathcal{X}}$. Therefore, the set $\{ y \in p^{-1}(x)~|~d_{\widetilde{\mathcal{X}}}(\tilde{x},y) \leq \varepsilon \}$ is compact as well. As it is also discrete, we conclude that it is finite, as required.
\end{proof}

Let $f \colon \mathcal{X}\to B$ be a surjective morphism of complex-analytic spaces. Choose a basepoint $x_0\in \mathcal{X}$ and define $b_0:=f(x_0)$. Let $\Sec(f):=\mathrm{Hol}_B(B,\mathcal{X})$ be the set of sections of $f$. Let $\ConSec(\pi_1(f))$ be the set of conjugacy classes of sections of the induced homomorphism $f_\ast\colon \pi_1(\mathcal{X},x_0)\to \pi_1(B,b_0)$. We define the \emph{Kummer map of $f$} to be the map of sets
\[ \Kum_f \colon \Sec(f) \to \ConSec(\pi_1(f)) \]
which associates to a section of $f$ its associated conjugacy class of sections of $\pi_1(f)$. 

The image of the Kummer map can be controlled using the finite generation of fundamental groups and the Kobayashi metric, assuming the fibres of $f$ are hyperbolic.

\begin{theorem}[Hyperbolic varieties] \label{theorem:hyperbolic_families}
Let $B$ be a  hyperbolic connected Riemann surface whose fundamental group $\pi_1(B)$ is finitely generated. Let $f \colon \mathcal{X} \to B$ be a proper surjective holomorphic map of complex-analytic spaces with connected fibres. Assume that, for every $b$ in $B$, the fibre $\mathcal{X}_b = f^{-1}(b)$ is hyperbolic. Then, the image of the Kummer map $\Kum_f$ is finite.
\end{theorem}
\begin{proof}
(This is \cite[Proposition~2]{ParshinLang}; see also the proof of \cite[Claim~4.4]{SzamuelyParshin}.)
Let $x\in \mathcal{X}$ and let $b:=f(x)$. Let $\gamma_1,\ldots,\gamma_n$ be paths that generate $\pi_1(B,b)$ and let $d$ denote the maximum of the Kobayashi-lengths of the $\gamma_i$. Let $d'$ be the Kobayashi-diameter of $\mathcal{X}_b$, and note that $d'$ is finite by compactness of $\mathcal{X}_b$. (Recall that $d'$ is the maximum of $d_{\mathcal{X}_b}(y,z)$, as $y$ and $z$ range over the compact metric space $\mathcal{X}_b$.) 
For every point $x'$ in $\mathcal{X}_b$, fix a path $L_{x'}$ from $x$ to $x'$ of Kobayashi-length at most $d'$. Then, if $\sigma \colon B \to \mathcal{X}$ is a section and $\gamma$ is a loop in $B$ based at $b$, the concatenation of the three paths $L_{\sigma(b)}^{-1}$ (where the inverse refers to reversing the path-direction), $\sigma \circ \gamma$, and $L_{\sigma(b)}$, is a loop in $\mathcal{X}$ based at $x$. This construction defines a morphism $\pi_1(B,b) \to \pi_1(\mathcal{X},x)$ and the class of this morphism in $\mathrm{ConSec}(\pi_1(f))$ is $\Kum_f(\sigma)$. Observe that this morphism sends every $\gamma_i$ to a loop of Kobayashi-length at most $d+2d'$.

 Since $f$ is proper with hyperbolic fibres, its fibres are complete hyperbolic (see \cite[Theorem~3.2.1]{KobayashiBook}).   Moreover, since a hyperbolic connected Riemann surface is complete hyperbolic (see Remark \ref{remark:RS_is_complete}), it follows from \cite[Corollary~3.11.2]{KobayashiBook} that $\mathcal{X}$ is complete hyperbolic. Lemma \ref{lemma:parshin_topological} now implies that there are only finitely many elements of $\pi_1(\mathcal{X},x)$ representable by a path of Kobayashi-length at most $d+2d'$. As there are only finitely many $\gamma_i$ and their images in $\pi_1(\mathcal{X},x)$ fully determine the morphism $\pi_1(B,b) \to \pi_1(\mathcal{X},x)$, we see that the elements of $\mathcal{X}(B)$ induce only finitely many morphisms $\pi_1(B,b) \to \pi_1(\mathcal{X},x)$. This concludes the proof.
\end{proof}

Parshin does not prove the finiteness of the fibres of $\Kum_f$ directly. Instead, he shifts his attention to abelian schemes (which is of course only helpful if $\mathcal{X}$ is closely related to one). To infer properties of $\Kum_f$ from the Kummer map of another family of varieties, we require the compatibility of Kummer maps. To explain this, let $B$ be a complex-analytic space and let $\varphi \colon \mathcal{X}\to \mathcal{Y}$ be a morphism of complex-analytic spaces over $B$. Assume $f \colon \mathcal{X}\to B$ and $g \colon \mathcal{Y}\to B$ have connected fibres. 
The map $\varphi$ induces a map of sets $\mathrm{Sec}(f) \to \mathrm{Sec}(g)$. Moreover, given a section of $\pi_1(f)$, post-composing it with the morphism $\pi_1(\varphi)\colon \pi_1(\mathcal{X})\to \pi_1(\mathcal{Y})$ gives an element of $\mathrm{ConSec}(\pi_1(g))$, and one easily checks that this gives rise to a map of sets $\mathrm{ConSec}(\pi_1(f)) \to \mathrm{ConSec}(\pi_1(g))$. These maps are easily seen to be compatible, i.e., the following diagram commutes.
\[
\xymatrix{ \mathrm{Sec}(f) \ar[rr] \ar[d]_{\mathrm{Kum}_f}& & \mathrm{Sec}(g) \ar[d]^{\mathrm{Kum}_g} \\ \mathrm{ConSec}(\pi_1(f)) \ar[rr] & & \mathrm{ConSec}(\pi_1(g))}
\]

The above results lead to the following criterion for finiteness of sections.

\begin{theorem}[Criterion]\label{thm:parshin}
Let $B$ be a smooth quasi-projective hyperbolic curve over $\mathbb{C}$. Let $f \colon \mathcal{X}\to B$ and $g \colon \mathcal{Y}\to B$ be proper morphisms of varieties over $\mathbb{C}$ with connected fibres. Let $\mathcal{X}\to \mathcal{Y}$ be a finite morphism over $B$. Let $B^o \subseteq B^{\an}$ be a non-empty connected Euclidean-open. Assume that
\begin{enumerate}
\item the restricted Kummer map $\Kum_{g^{\an}}|_{\mathcal{Y}(B)}\colon \mathcal{Y}(B)\to \mathrm{ConSec}(\pi_1(g^{\an}))$ has finite fibres, 
\item the natural morphism $\pi_1(B^o)\to \pi_1(B^{\an})$ is an isomorphism, and 
\item for every $b$ in $B^o$, the fibre $\mathcal{X}_b$ is hyperbolic. 
\end{enumerate}
Then $\mathcal{X}(B)$ is finite.
\end{theorem}
\begin{proof} 
Define $\mathcal{X}^o := (f^{\an})^{-1}(B^o)$. As the induced morphism $f^o \colon \mathcal{X}^o \to B^o$ of complex-analytic spaces has connected fibres, we may consider its Kummer map.
Since the natural map $\pi_1(B^o)\to \pi_1(B^{\an})$ is an isomorphism, the group homomorphism $\pi_1(\mathcal{X}^o) \to \pi_1(\mathcal{X}^{\an})$ induces a map of sets $\ConSec(\pi_1(f^o)) \to \ConSec(\pi_1(f^{\an}))$. We obtain the following commutative diagram.
\[ \xymatrix{
\mathrm{Sec}(f^o) \ar[d]_{\mathrm{Kum}_{f^o}} & & \mathcal{X}(B) \ar[ll]_{\textrm{inclusion}} \ar[rr] \ar[d]_{\mathrm{Kum}_{f^{\an}}} & & \mathcal{Y}(B) \ar[d]^{\mathrm{Kum}_{g^{\an}}} \\
\mathrm{ConSec}(\pi_1(f^o)) \ar[rr] & & \mathrm{ConSec}(\pi_1(f^{\an})) \ar[rr] & & \mathrm{ConSec}(\pi_1(g^{\an}))}
\]
As $\pi_1(B^{\an})$ is finitely generated and $\pi_1(B^o)$ is isomorphic to $\pi_1(B^{\an})$, the group $\pi_1(B^o)$ is finitely generated. As $B$ is hyperbolic and $B^o$ is open in $B^{\an}$, we have that $B^o$ is hyperbolic. Thus, by Theorem \ref{theorem:hyperbolic_families}, the Kummer map $\Kum_{f^o}$ has finite image. By the commutativity of the above diagram, it follows that $\Kum_{f^{\an}}$ has finite image as well.

Since $\mathcal{X}\to \mathcal{Y}$ is finite, the map $\mathcal{X}(B)\to \mathcal{Y}(B)$ has finite fibres.
Moreover, by assumption, the restricted Kummer map $\Kum_{g^{\an}}|_{\mathcal{Y}(B)} \colon \mathcal{Y}(B)\to \mathrm{ConSec}(\pi_1(g^{\an}))$ has finite fibres. Therefore, the composed map
\[ \xymatrix{\mathcal{X}(B) \ar[rr] & & \mathcal{Y}(B) \ar[rr]^-{\Kum_{g^{\an}}|_{\mathcal{Y}(B)}} & & \mathrm{ConSec}(\pi_1(g^{\an}))} \]
has finite fibres. It follows from the above diagram that the restricted Kummer map
\[ \Kum_{f^{\an}}|_{\mathcal{X}(B)} \colon \mathcal{X}(B) \to \mathrm{ConSec}(\pi_1(f^{\an})) \]
also has finite fibres. Since it also has finite image, we conclude that $\mathcal{X}(B)$ is finite.
\end{proof}

Theorem \ref{thm:parshin} implies that the geometric Bombieri--Lang conjecture can be proven by controlling the fibres of an auxiliary Kummer map. The latter is possible for traceless families of abelian varieties using the Lang--N\'eron theorem. (Note that the following result does not guarantee the finiteness of the fibres of the Kummer map, but rather of the Kummer map restricted to the subset of \emph{algebraic} sections.)

\begin{theorem}[Abelian varieties] \label{thm:ln}
Let $B$ be a smooth curve over $\mathbb{C}$ with function field $K = K(B)$. Let $g\colon \mathcal{A}\to B$ be an abelian scheme. If the $K/\mathbb{C}$-trace of $\mathcal{A}_K$ is zero, then the restricted Kummer map $\Kum_{g^{\an}}\colon \mathcal{A}(B)\to \mathrm{ConSec}(\pi_1(g^{\an}))$ has finite fibres.
\end{theorem}
\begin{proof}
This is \cite[p.~171, Proposition~1]{ParshinLang}.
\end{proof}

We now come to the proof of the geometric Bombieri--Lang conjecture for projective varieties with a finite morphism to a traceless abelian variety:

\begin{proof}[Proof of Theorem \ref{thm:xy}]
Note that it suffices to prove that $X(K)$ is finite. Let $B$ be a smooth quasi-projective curve with function field $K = K(B)$ and choose a proper surjective morphism $f \colon \mathcal{X}\to B$ of varieties over $\mathbb{C}$ such that $\mathcal{X}_K \cong X$ over $K$. We have that $X(K) = \mathcal{X}(B)$. To prove the finiteness of $X(K)$, we may replace $K$ by a finite field extension. Moreover, we may replace $B$ by a dense open whenever necessary. Thus, we may assume that
\begin{enumerate}
\item there is an abelian variety $A$ over $K$ such that the $\overline{K}/\mathbb{C}$-trace of $A_{\overline{K}}$ is zero,
\item there is a finite morphism $X\to A$ over $K$, 
\item the genus of the smooth projective model $\overline{B}$ of $B$ is at least two (so that $B$ is hyperbolic),
\item there is an abelian scheme $g\colon \mathcal{A}\to B$ with $\mathcal{A}_K\cong A$ over $K$, and
\item there is a finite morphism $\mathcal{X}\to \mathcal{A}$ extending the morphism $X\to A$ over $B$.
\end{enumerate}

Let $\overline{\mathcal{X}} \to \overline{B}$ be a proper surjective morphism extending $\mathcal{X} \to B$. Then, by Corollary \ref{cor:Yamanoi_groupless_family}, the set of $b$ in $\overline{B}(\mathbb{C})$ with $\overline{\mathcal{X}}_b$ not hyperbolic is a Euclidean-closed countable subset $S \subseteq \overline{B}^{\an}$. We may choose finitely many pairwise disjoint closed disks $D_i \subseteq \overline{B}^{\an}$, one centered around each point of $\overline{B} \setminus B$, such that the boundaries $\partial D_i$ are disjoint from $S$. Noting that $S$ is compact, we may choose finitely many more disks until every point of $S$ is contained in one of these disks (retaining that the disks are pairwise disjoint). Let $d_1, \ldots, d_r$ denote the centers of the disks, and note that by construction, we have that $\overline{B} \setminus \{d_1, \ldots, d_r\} \subseteq B$. Thus, we may replace $B$ by $\overline{B} \setminus \{d_1, \ldots, d_r\}$ without affecting (i) to (v) above. Let $B^o := \overline{B}^{\an} \setminus (D_1 \cup \ldots \cup D_r)$ and observe that by construction, we have that the fibre $\mathcal{X}_b$ is hyperbolic for every $b \in B^o$ and moreover that the inclusion map $B^o \subseteq B^{\an}$ is a homotopy equivalence. Now, since the restricted Kummer map $\Kum_{g^{\an}}\colon \mathcal{A}(B)\to \mathrm{ConSec}(\pi_1(g^{\an}))$ has finite fibres (Theorem \ref{thm:ln}), we conclude by Theorem \ref{thm:parshin}.
\end{proof}

\bibliography{parshin_refs}{}
\bibliographystyle{alpha}

\end{document}